\documentclass[reqno]{amsart}
\usepackage {amsmath, amscd} 
\usepackage {amssymb}
\usepackage {epsfig}
\usepackage {bm}
\usepackage {indentfirst} %indent the first par after section
\usepackage{color}

\numberwithin{equation}{section}

\def\<{\langle}
\def\>{\rangle}

\def\EE{{\mathcal E}}

\def\LL{{\mathcal L}}

\def\bbC{\mathbb{C}}
\def\bbZ{\mathbb{Z}}
\def\bbN{\mathbb{N}}

\def\bbD{\mathbb{D}}
\def\bbT{\mathbb{T}}

\newtheorem{lemma}{Lemma}[section]
\newtheorem{proposition}[lemma]{Proposition}%[section]
\newtheorem{theorem}[lemma]{Theorem}%[section]
\newtheorem{corollary}[lemma]{Corollary}

\theoremstyle{definition}
\newtheorem{remark}[lemma]{Remark}

\title{Unitary equivalence to truncated Toeplitz operators}
\author{E. Strouse}
\address{ Universit\'e de Bordeaux, Institut de Math\'ematiques de Bordeaux UMR 5251,
351, cours de la Lib\'eration,
F-33405 Talence cedex, France}
\email{Elizabeth.Strouse@math.u-bordeaux1.fr}

\author{D. Timotin}
\address{Institute of Mathematics of the Romanian Academy, P.O. Box 1-764, Bucharest 
014700, Romania}
\email{Dan.Timotin@imar.ro}

\author{M. Zarrabi}
\address{ Universit\'e de Bordeaux, Institut de Math\'ematiques de Bordeaux UMR 5251,
351, cours de la Lib\'eration,
F-33405 Talence cedex, France}
\email{Mohamed.Zarrabi@math.u-bordeaux1.fr}


\begin{document}

\begin{abstract}
In this paper we investigate operators unitarily equivalent to truncated
Toeplitz operators. We show that this class contains  certain sums of tensor products of truncated Toeplitz operators. In particular, it contains arbitrary inflations of truncated Toeplitz operators; this answers a question posed in~\cite{CGRW}. 
\end{abstract}

\maketitle

\section{Introduction}

Truncated Toeplitz operators, defined on certain subspaces of the
$H^2$ called {\it model spaces}, are a generalization of the operators
associated with Toeplitz matrices. They are introduced and discussed in great detail in a
recent survey paper by Sarason~\cite{S}. Some special cases
have appeared long ago in the literature: the  model operators for
contractions with defect number one as well as their commutant are truncated Toeplitz operators with analytic symbols (see, for instance,~\cite{S1, SNF, N}). This is a new 
area of study, and  many simple questions 
remain open. The basic reference for this subject is~\cite{S}, subsequent
work 
is done in~\cite{CGRW,Cima-2008,BCFMT, BBK}.

The truncated Toeplitz operators live on the model spaces $K_\Theta=H^2\ominus \Theta H^2$, with $\Theta$ an inner function. These are the subspaces of $H^2$ invariant for the adjoint of multiplication by the variable $z$ (see Section~\ref{se:preliminaries} 
for precise definitions). Given a model space $K_\Theta$ and a function $\phi\in L^2$, the truncated 
Toeplitz operator $A^\Theta_\phi$ is  defined on a dense subspace of 
$K_\Theta$ as the compression to $K_\Theta$ of multiplication by $\phi$. 
The function $\phi$ is then called a symbol of the operator. For $\Theta = z^n$ one obtains the usual Toeplitz matrices, but
consideration of such operators for arbitrary inner functions brings
up many new and interesting questions.

In this paper we study unitary
equivalence to truncated Toeplitz operators. This problem is discussed at length in~\cite{CGRW}, where it is shown, in particular, that all rank
one operators, two by two matrices, and normal operators are unitarily
equivalent to truncated Toeplitz operators. Thus the class of operators unitarily equivalent to truncated Toeplitz operators is larger than might be expected. In~\cite{CGRW} the authors ask several questions concerning this class; in particular, whether it contains inflations and direct sums of truncated Toeplitz operators.

The purpose of this paper is to continue along this line of investigation. By obtaining  new classes of operators unitarily equivalent to truncated Toeplitz operators,  we answer in particular a question in~\cite{CGRW} concerning inflations, and provide some partial results related to direct sums. The methods used are mostly based on analyzing the action of composition by inner functions on model spaces and on truncated Toeplitz operators. It is suggested by the proof of Theorem 5.8 in~\cite{CGRW}; a similar technique appears, in the case of Toeplitz operators on $H^2$, in a much earlier paper of Cowen~\cite{Co}. It should be noted that little is known about unitary equivalence to classical Toeplitz operators or to Toeplitz matrices.

The plan of the paper is the following: the next section is devoted to preliminaries. We introduce the notations and obtain some basic results concerning the effect on model spaces of composition by an inner function. Section~3 contains the main  result, which is applied to some particular situations in Section~4. Finally, Section~5 is devoted to a special type of direct sum, investigating when the direct sum of a truncated Toeplitz operator with the zero operator is unitarily equivalent to a truncated Toeplitz operator.

\section{Model spaces and composition}\label{se:preliminaries}

We start by introducing the main objects and results that we shall use; convenient references are~\cite{N} for model spaces and~\cite{S} for TTOs.

Let $H^2$ be the classical Hardy space of holomorphic functions on the
open unit disc $\bbD$ having square-summable Taylor coefficients at
the origin. By considering radial limits  almost everywhere with respect to Lebesgue measure $dm$ on $\partial\bbD=\bbT$,  $H^2$ can be identified with the closed subspace of
$L^2=L^2( \bbT, dm)$ consisting of functions with negative Fourier
coefficients equal to 0; its orthogonal complement $L^2\ominus H^2$ is denoted by $H^2_-$. An element $\Theta \in H^2$ is called an inner function if 
$|\Theta(e^{it})| = 1$ for almost all $t\in [0,2\pi]$; an
inner function $\Theta$ is of order $n$ if $\Theta$ is a finite Blaschke product
of degree $n$, and of order infinity otherwise.

Let $S$ denote  multiplication by $z$ on
$H^2.$  The closed subspaces of $H^2$ invariant  for the adjoint $S^\ast$ of $S$ are precisely of the form $K_\Theta=H^2\ominus \Theta H^2$, with $\Theta$ inner; they are often called \emph{model spaces}. In general we will use the notation $P_Y$ for the orthogonal projection onto a given subspace $Y$ of $L^2$;  we abbreviate $P_{H^2}$ to $P_+$, $P_{(H^2)^\perp}$ to $P_-$, and
$P_{K_\Theta}$ to $P_\Theta$.  The map $f\mapsto J(f)=\Theta \bar z \bar f$ is an anti-isometry of $K_\Theta$ onto itself; in particular, we will use in the proof of Proposition~\ref{pr:Omega_B and K_Theta} the fact that if $f\in K_\Theta$, then $\Theta \bar f\in zH^2$.

For $\phi \in L^2$,  the (bounded or unbounded) \emph{truncated Toeplitz operator} (TTO) $A^\Theta_\phi$ is defined on $K_\Theta$ by $A^\Theta_\phi f=P_\Theta  \phi f$. Note that since $H^\infty\cap K_\Theta$ is dense in $K_\Theta$ (\cite{A}, see also~\cite{CMR}), $A_\phi^\Theta$ always has a dense domain. The function $\phi$ is called a \emph{symbol} of the operator. It is not uniquely defined; more precisely, 
\begin{equation}\label{eq:same symbol}
A_\phi^\Theta = A_\psi^\Theta\quad   \Longleftrightarrow\quad \phi-\psi\in \Theta H^2+ \overline{\Theta H^2}.
\end{equation}

Since tensor products appear quite often in the paper, we will use a different notation for rank one operators, namely
\[
(x\odot y)(\xi)=\<\xi,y\>x.
\]

For  $\lambda\in \bbD$ one defines
\begin{equation}\label{eq:reproducing kernels}
k^\Theta_\lambda(z)=\frac{1-\overline{\Theta(\lambda)}\Theta(z)}{1-\bar\lambda z};\quad \tilde k^\Theta_\lambda(z)= (J(k^\Theta_\lambda))(z)=\frac{\Theta(z)-\Theta(\lambda)}{z-\lambda} .
\end{equation}
The functions $k^\Theta_\lambda$ are the reproducing kernels in $K_\Theta$ for $\lambda\in \bbD$. In the case when the angular derivative of $\Theta$ at $\zeta\in\bbT$ exists, then formulas~\eqref{eq:reproducing kernels}, with $\lambda$ replaced by $\zeta$, define functions in $K_\Theta$,  evaluation in $\zeta$ is a continuous functional on $K_\Theta$, and $k^\Theta_\zeta$ is the corresponding reproducing kernel.

The reproducing kernels produce the rank one TTOs in $K_\Theta$. More precisely, it is shown in ~\cite[Theorem 5.1]{S} that these are exactly the scalar multiples of $\tilde k^\Theta_\lambda\odot k^\Theta_\lambda$, $ k^\Theta_\lambda\odot \tilde k^\Theta_\lambda$ for $\lambda\in \bbD$, and $k^\Theta_\zeta \odot k^\Theta_\zeta$ for the points $\zeta \in \bbT$ where the angular derivative of $\Theta$ exists. We will use below the precise form of some symbols (which appears also in~\cite{S}): namely, $\phi=\bar z \Theta$ is a symbol for $\tilde k^\Theta_0\odot k^\Theta_0$ and, for $\zeta\in\bbT$, $\phi=k^{\Theta}_\zeta + \overline {k^{\Theta}_\zeta}-1$ is a symbol for $k^{\Theta}_\zeta\odot k^{\Theta}_\zeta$.

\begin{lemma}\label{le:if one is bounded then the other is also}
Suppose that $\Theta$ is an inner function and that $\phi\in L^2$. If $A^\Theta_\phi$ is bounded, then $A^\Theta_{z^j\phi}$ is also bounded for any $j\in \bbZ$.
\end{lemma}

\begin{proof} Suppose first that $j=1$.
Denote by $P_\Theta^0$ the projection onto $K_\Theta\ominus \bbC \tilde k^\Theta_0$. Take $f\in K_\Theta$. We have
\begin{equation}\label{eq:11}
A^\Theta_{z\phi} f= A^\Theta_{z\phi}(P_\Theta^0 f + (I-P_\Theta^0)f)=
P_\Theta(z\phi P_\Theta^0 f )+P_\Theta (z\phi (I-P_\Theta^0)f).
\end{equation}

One checks easily that, if $g\in K_\Theta$ and $g\perp \tilde k^\Theta_0$, then $zg\in K_\Theta$.
Therefore $zP^0_\Theta f\in K_\Theta$, and thus
\[
P_\Theta (z\phi P_\Theta^0 f) =P_\Theta\phi (z P_\Theta^0 f) =A^\Theta_\phi (zP_\Theta^0 f).
\]
Since
$A^\Theta_\phi$ is bounded by assumption, $ P_\Theta z\phi P_\Theta^0$ is bounded.
On the other hand, $(I-P_\Theta^0)$, the projection of rank~1 onto $\bbC \tilde k^\Theta_0 $, has the formula $h^\Theta_0 \odot h^\Theta_0$, where we have denoted $h^\Theta_0=\frac{\tilde k^\Theta_0}{\|\tilde k^\Theta_0\|}$.  Since $z\phi  h^\Theta_0\in L^2$, $z\phi  h^\Theta_0\odot h^\Theta_0$ is a bounded rank one operator, and we may write
\[
P_\Theta z\phi (I-P_\Theta^0)f=
P_\Theta z\phi \<f, h^\Theta_0\> h^\Theta_0=
P_\Theta\left( z\phi  h^\Theta_0\odot h^\Theta_0\right) f.
\]
It follows then from~\eqref{eq:11} that $A^\Theta_{z\phi}$ is bounded.

The statement is thus proved for $j=1$. It is obvious that we may apply induction to obtain $A^\Theta_{z^j \phi}$  bounded for any $j\in\bbN$.
But if $A^\Theta_\phi$ is bounded then $A^\Theta_{\bar\phi}=(A^\Theta_\phi)^*$ is  bounded, and therefore also $A^\Theta_{z^j\phi}=(A^\Theta_{z^{-j}\bar\phi})^*$ for $j\le 0$. The lemma is proved.
\end{proof}

The result of the next proposition is implicitly contained in the proof of~\cite[Theorem 1]{Co}.

\begin{proposition}\label{pr:Omega_B}
Suppose $B$ is an inner function. Then the formula
\[
 h\otimes f\mapsto h(f\circ B)
\]
can be extended linearly to define a unitary operator $\Omega_B$ from $K_B\otimes L^2$ onto $L^2$. The operator $\Omega_B$ maps $K_B\otimes H^2$ onto $H^2$.
\end{proposition}

\begin{proof} 
 To show that the given map is isometric, it is enough to check that it maps an orthonormal basis of $K_B\otimes L^2$ to an orthonormal family of $L^2$. Let  $(h_i)_{i\in I}$
     be an orthonormal basis for $K_B$; then  $(h_i \otimes z^j)_{i\in I, j\in\bbZ}$ is a basis for $K_B\otimes
 L^2$, and we have to show that
\[
 \<h_i B^j, h_{i'} B^{j'}\>=\delta_{i,i'}\delta_{j,j'}.
\]
If $j = j^{'}$, then 
\[
 \<h_i B^j, h_{i'} B^{j}\>=\delta_{i,i'}\]
since multiplication by an inner function is isometric. On the other hand, if $j\not=j'$, assume, for instance, that
 $j<j^{'}$; then
 \[
 \<h_k B^j, h_{k'} B^{j'}\>= \<h_k, h_{k'} B^{j'-j}\>
\]
which equals zero since the first term is in $K_B$, while the second
is in $BH^2$.

We have $\Omega_B(K_B\otimes \bbC z^j)=B^j K_B$. Since $L^2=\bigoplus_{j\in \bbZ}B^j K_B$, it follows that $\Omega_B$ is onto and thus unitary. Also, $H^2=\bigoplus_{j\in \bbN}B^j K_B$, and thus $\Omega_B$ maps $K_B\otimes H^2$ precisely onto $H^2$. 
\end{proof}

We will now consider  a second inner function $\Theta$. Remarkably, the unitary operator of Proposition~\ref{pr:Omega_B} behaves well with respect to $K_\Theta$.

\begin{proposition}\label{pr:Omega_B and K_Theta}
Suppose that $B$ and  $\Theta$ are inner functions, while $\Omega_B$ is the unitary defined in Proposition~\ref{pr:Omega_B}. Then 
\[
\Omega_B(K_B\otimes \Theta H^2)=(\Theta\circ B)H^2, \qquad  \Omega_B(K_B\otimes K_\Theta)=K_{\Theta\circ B}.
\]
\end{proposition}

\begin{proof}
It is enough to prove the inclusions  $\Omega_B(K_B\otimes \Theta H^2)\subset(\Theta\circ B)H^2$ and $ \Omega_B(K_B\otimes K_\Theta)\subset K_{\Theta\circ B}$, since $\Omega_B$ is unitary and $H^2= (\Theta\circ B)H^2\oplus K_{\Theta\circ B}$.

If we fix $h\in K_B$ and $j\in\bbN$, we have
\[
 \Omega_B(h\otimes \Theta z^j)=h(\Theta\circ B)B^j\in (\Theta\circ B)H^2, 
\]
 and thus the first inclusion is true.

For the second inclusion, let $f$ be another element of $ K_\Theta$. We need to see that
\[
 \<h (f\circ B), (\Theta\circ B)z^k\>=0
\]
for all $k\ge 0$. This equality can be written as
\[
 \< S^*{}^k h, (\Theta \bar f)\circ B\>=0,
\]
where $S$ is the shift operator on $H^2$ ($Sf=zf$). But, if $f\in K_\Theta$, then $\Theta\bar f\in zH^2$, and therefore $(\Theta \bar f)\circ B\in BH^2$. Since $h\in K_B$ and $S^* K_B\subset K_B$, the scalar product above is indeed null.
\end{proof}

We will denote by $\omega_B:K_B\otimes K_\Theta\to K_{\Theta\circ B}$ the restriction of $\Omega_B$ to $K_B\otimes K_\Theta$.

\section{Truncated Toeplitz operators}

We start with a lemma about the action of certain multiplications.

\begin{lemma}\label{le:general intertwining}
Suppose that $B$ is an inner function,
$\psi,  \phi\in L^2$, $h\in K_B\cap H^\infty$, $f\in L^\infty$, and that the  operators $ A^B_{{\bar B}^j\psi}$ are nonzero  only for a finite number of $j\in\bbZ$.
Then
\[
\psi(\phi\circ B)\Omega_B (h\otimes f)= \Omega_B
\big(\sum_j (A^B_{{\bar B}^j\psi}h\otimes z^j\phi f)\big).
\]
\end{lemma}

\begin{proof}
We have $\psi(\phi\circ B)\Omega_B (h\otimes f)=\psi (\phi\circ B)h(f\circ B)$.
The assumptions imply that $\psi h\in L^2=\bigoplus_{j\in \bbZ}B^j K_B$. Since 
\[
P_{B^jK_B}(\psi h)=B^jP_{K_B}\bar B^j \psi h= B^j A^B_{{\bar B}^j\psi}h,
\]
and $ A^B_{{\bar B}^j\psi}\not=0$ only for a finite number of $j$, it follows that in the orthogonal decomposition $\psi h=\sum_{j}P_{B^jK_B}(\psi h)$ the sum is finite. Thus, we can write
\[
\begin{split}
\psi h  (\phi\circ B)(f\circ B)&=\left(\sum_{j}P_{B^jK_B}(\psi h)\right)(\phi\circ B)(f\circ B)\\
&=\sum_{j} B^jP_{B}({\bar B}^j\psi h)(\phi\circ B)(f\circ B)\\
&=\sum_{j} \Omega_B \left(A^B_{{\bar B}^j\psi}h\otimes z^j\phi f\right),
\end{split}
\]
which finishes the proof.
\end{proof}

\begin{theorem}\label{th:main}
 Let $B$ and $\Theta$ be inner functions, and suppose that
 $\psi,\phi\in L^2$ are subjected to the following  conditions:
 
(a) The operators $ A^B_{{\bar B}^j\psi}$ are bounded and nonzero only for a finite number of $j\in\bbZ$.

(b) $ A^\Theta_{\phi}$ is bounded.

(c) $\psi(\phi\circ B)\in L^2$.

Then $A^{\Theta\circ B}_{\psi(\phi\circ B)}$ is bounded, and
\begin{equation}\label{eq:main}
 A^{\Theta\circ B}_{\psi( \phi\circ B)}\omega_B=
 \omega_B \left( \sum_j (A^B_{{\bar B^j}\psi}\otimes A^\Theta_{z^j\phi}) \right).
\end{equation}
\end{theorem}

\begin{proof}
Let us note first that, since $\Omega_B(K_B\otimes K_\Theta)=K_{\Theta\circ B}$, we have 
\[
P_{\Theta\circ B}=\Omega_B (P_B\otimes P_\Theta)\Omega_B^*.
\]
Let $h\in K_B\cap H^\infty$, $f\in K_\Theta\cap H^\infty$. We have, using Lemma~\ref{le:general intertwining},
\begin{align*}
A^{\Theta\circ B}_{\psi( \phi\circ B)}\omega_B (h\otimes f)&
=P_{\Theta\circ B} \psi(\phi\circ B) \Omega_B (h\otimes f)
=\Omega_B (P_B\otimes P_\Theta)\Omega_B^*\psi(\phi\circ B) \Omega_B    (h\otimes f) \\
&=\Omega_B (P_B\otimes P_\Theta) \big(\sum_j (A^B_{{\bar B}^j\psi}h\otimes z^j\phi f)\big)\\
&=\omega_B  \big(\sum_j (A^B_{{\bar B}^j\psi}h\otimes A^\Theta_{z^j\phi} f)\big).
\end{align*}
By Lemma~\ref{le:if one is bounded then the other is also}, $A^\Theta_{z^j\phi}$ is bounded for any~$j$. Therefore the operator on the right hand side of the last equation is bounded by assumption; it follows that $A^{\Theta\circ B}_{\psi(\phi\circ B)}$ is also bounded, and~\eqref{eq:main} is true. The theorem is proved.
 \end{proof}

\begin{remark}\label{re:sufficient conditions}

\begin{enumerate}
\item If $\psi\in B^k K_B$ for some $k\in \bbZ$, then $A^B_{{\bar B}^j\psi}=0$ for $j\not=k, k+1$, while $\psi(\phi\circ B)\in L^2$ by  Proposition~\ref{pr:Omega_B}. Therefore, if   $\psi\in \oplus_{j\in J}B^j K_B$ for a finite set $J$, then $ A^B_{{\bar B}^j\psi}$ is nonzero only for a finite number of $j\in\bbZ$, while $\psi(\phi\circ B)\in L^2$. Thus part of the hypotheses of Theorem~\ref{th:main} are automatically satisfied.

\item The boundedness conditions in the statement of Theorem~\ref{th:main} are immediately satisfied if $\phi,\psi\in H^\infty$. It is known however~\cite{BBK, BCFMT} that there exist bounded TTOs that have no bounded symbols, and therefore such an assumption would reduce the generality of the result. In~\cite{BBK} one characterizes precisely the inner functions $\Theta$ that have the property that any bounded TTO on $K_\Theta$ has a bounded symbol.
\end{enumerate}

\end{remark}

We will discuss in the next section several applications of Theorem~\ref{th:main}; but first we give some simple but important consequences. 
 
\begin{corollary}\label{co:immediate corollary}
Suppose that $\Theta$ is an inner function, that $\phi\in L^2$, and that $B$ is an inner function of order $k$ for some $k=1,2\dots,\infty$. Assume also that $A^\Theta_\phi$ is bounded.
Then $A^{\Theta\circ B}_{ \phi\circ B}$ is bounded and unitarily equivalent to  $I_k\otimes A^\Theta_\phi$.
\end{corollary}

\begin{proof}
Take $\psi\equiv1$. We have $A^B_\psi=I_{K_B}$ and $A^B_{\bar B^j\psi}=0$ for $j\not=0$, whence  (a) in Theorem~\ref{th:main} is satisfied. (b) is satisfied by assumption, while (c) is true since $\phi\circ B\in L^2$ by Littlewood's Subordination Theorem. Therefore $A^{\Theta\circ B}_{ \phi\circ B}$ is bounded, and~\eqref{eq:main} becomes
\begin{equation*}
A^{\Theta\circ B}_{ \phi\circ B}\omega_B=\omega_B (I_{K_B}\otimes A^\Theta_\phi),
\end{equation*}
which proves the corollary.
\end{proof}

For any $\Theta$ and $k$, if we take $B$ to be an inner function  of order $k$, then  Corollary~\ref{co:immediate corollary} answers Question 5.10 of~\cite{CGRW}. 

The next immediate consequence is the analogue for TTOs of a result for usual Toeplitz operators from the paper~\cite{Co} mentioned in the introduction.

\begin{corollary}\label{co:same order}
Suppose that $\Theta$ is an inner function and that $\phi\in L^2$, such that $A^\Theta_\phi$ is bounded.
If $B_1, B_2$ are two inner functions of the same order, then $A^{\Theta\circ B_1}_{ \phi\circ B_1}$ and $A^{\Theta\circ B_2}_{ \phi\circ B_2}$ are bounded and unitarily equivalent.
\end{corollary}

\section{Applications}

We intend to apply Theorem~\ref{th:main} to 
different choices of $B$ and/or $\Theta$, in order to find several classes of operators which are unitarily equivalent to TTOs. In all  the examples in this section we choose $\psi\in K_B$. Then the assumptions of Theorem~\ref{th:main} are satisfied (see Remark~\ref{re:sufficient conditions}~(1)); moreover,  $A^B_{\bar B^j\psi}=0$ for all $j\not=0,1$. Thus we have the following corollary of Theorem~\ref{th:main}, which deserves to be mentioned separately.

\begin{corollary}\label{co:psi in K_B}
If $\psi\in K_B$ and $A^B_\psi, A^B_{\bar B\psi}, A^\Theta_\phi$ are bounded,   then $A^{\Theta\circ B}_{\psi( \phi\circ B)}$ is bounded and
\begin{equation}\label{eq:main2}
 A^{\Theta\circ B}_{\psi( \phi\circ B)}\omega_B=\omega_B [A^B_\psi\otimes A^\Theta_\phi + A^B_{\bar B\psi}\otimes A^\Theta_{z\phi}].
\end{equation}
\end{corollary}

\subsection{}\label{sse:1}
 Suppose $B(z)=z^n$. Then $K_B$ is equal to the set of analytic polynomials of degree at most $k-1$, and  the monomials $\psi_m(z)=z^m$, $m=0,\dots,k-1$ form a basis of $K_B$.  If $A=(a_{ij})$ is the matrix of $A^B_{\psi_m}$ with respect to this basis, then 
\[
a_{ij}=\begin{cases}
1 & \text{ if }i=j+m,\\
0& \text{ otherwise,}
\end{cases} 
\]
and $A^B_{\bar B \psi_m}= (A^B_{\psi_{n-m}})^*$.

Let $\phi_m\in L^2$, $m=0,\dots,n-1$, such that all the operators $A^\Theta_{\phi_m}$ are bounded. Define
\[
h(z) = \sum_{m=0}^{n-1} z^m (\phi_m (z^n)) = \sum_{m=0}^{n-1} z^m (\phi_m \circ B) .
\]
Using the natural identification
of $K_B\otimes K_\Theta$  with the direct sum of $n$ copies of $K_\Theta$, formula~\eqref{eq:main2} and the above remarks imply that
$A^{\Theta\circ B}_h= \sum_{m=0}^{n-1} A^{\Theta\circ B}_{z^m
      (\phi_m \circ B)}$ is bounded and unitarily equivalent to the block Toeplitz operator matrix
\[
\begin{pmatrix}
A^\Theta_{\phi_0} & A^\Theta_{z\phi_{n-1}}& \dots &A^\Theta_{z\phi_1}\\
A^\Theta_{\phi_1} & A^\Theta_{\phi_0} & \dots& A^\Theta_{z\phi_2}\\
\ddots & \ddots & \ddots & \ddots\\
A^\Theta_{\phi_{n-1}}& A^\Theta_{\phi_{n-2}}& \dots& A^\Theta_{\phi_0}
\end{pmatrix}
\]
It is easily seen that, for $k<\infty$, Corollary~\ref{co:immediate corollary} is a particular case of the above formula.

\subsection{}\label{sse:2}

Next, suppose that  $\Theta(z)=z^n$. An argument similar
    to the one above  gives us yet another class of block Toeplitz
    operator matrices which are equivalent to truncated Toeplitz
    operators. This time, we take functions $\psi_i \in K_B$ ($-n \le i\le n-1$), and we assume that $A^B_{\psi_i}$ is bounded for $-(n-1)\le i \le n-1$, while $A^B_{\bar B\psi_i}$ is bounded for $-n\le i \le n-2$.
Consider a symbol function of the form
\begin{equation}{\label{eq:example2}}
h(z) = \sum_{m = -n}^{n-1} \psi_m(z)B^m
\end{equation}
Then the
operator $A^{B^n}_h$ is bounded and unitarily equivalent to the
following block Toeplitz operator
matrix, each entry of which is a truncated Toeplitz
operator on~$K_B $:
\begin{equation}\label{eq:sse2}
 \begin{pmatrix}
A^B_{\psi_0+\bar B\psi_{-1}} & A^B_{\psi_{-1}+\bar B\psi_{-2}}& \dots &A^B_{\psi_{-(n-1)}+\bar B\psi_{-n}}\\
A^B_{\psi_1+\bar B\psi_0} & A^B_{\psi_0+\bar B\psi_{-1}} & \dots& A^B_{\psi_{-(n-2)}+\bar B\psi_{-(n-1)}}\\
\ddots & \ddots & \ddots & \ddots\\
A^B_{\psi_{n-1}+\bar B\psi_{n-2}}& A^B_{\psi_{n-2}+\bar B\psi_{n-3}}& \dots& A^B_{\psi_0+\bar B\psi_{-1}}
\end{pmatrix}
\end{equation}

If $A = (a_{i-j})_{i,j =1,\dots, n}$  is a
    classical 
Toeplitz matrix, then, by
    considering the function $h$ in
    \eqref{eq:example2}  with $\psi_i = a_i$ we obtain (noting that $A^B_{\bar B}=0$) that
    $A^{B^n}_{\psi(B)}$ is unitarily equivalent to $A \otimes I_{K_B}$
      for any inner function $B$. This is
Theorem~5.8 of~\cite{CGRW}.

%There is a different way of looking at the matrix~\eqref{eq:sse2}. Namely, it is the general form of a block Toeplitz operator matrix $(A^B_{\Psi_{i-j}})_{i,j=1}^n$ with entries TTOs on $K_B$ subjected to the condition
%\begin{equation}\label{eq:generalform}
%P_-\Psi_i=\bar B(P_+ \Psi_{i-1}).
%\end{equation}
%Obviously the entries of the matrix in~\eqref{eq:sse2} satisfy~\eqref{eq:generalform}. On the other hand, if we are given a matrix of the form $(A^B_{\Psi_{i-j}})_{i,j=1}^n$ and the $\psi_i$ satisfy~\eqref{eq:generalform}, 
%then the matrix is actually of the
%  form~\eqref{eq:sse2}
%for the functions $\psi_i$ ($|i|\le n-1$) defined by $\psi_i=P_{K_B}\Psi_i$. Indeed, note first that $P_+=P_{K_B}+P_{B H^2}$ implies $\bar B(P_+ \Psi_{i-1})= \bar B(P_{K_B} \Psi_{i-1})+ \bar B(P_{BH^2} \Psi_{i-1})$. But, since the second term in the sum is in $H^2$, the assumption~\eqref{eq:generalform} implies that it is zero; so
%\[
%P_-\Psi_i=\bar B(P_{K_B} \Psi_{i-1}).
%\]
%But we also have $A^B_{P_+\Psi_i}=  A^B_{P_{K_B}\Psi_i}$, and so
%\[
%A^B_{\Psi_i}=A^B_{P_+\Psi_i}+A^B_{P_-\Psi_i}= A^B_{P_{K_B}\Psi_i}+ A^B_{\bar B P_{K_B} \Psi_{i-1}}=A^B_{\psi_i+\bar B \psi_{i-1}},
%\]
%which proves the claim.
%\marginpar{Could one replace~\eqref{eq:generalform} with a condition on operators?}
%

\subsection{}\label{sse:3}

Suppose that $B, \Theta$ are arbitrary inner functions and that $\phi=\bar z\Theta$. As noted in Section~\ref{se:preliminaries}, $A^\Theta_\phi=\tilde k^\Theta_0\odot k^\Theta_0$. On the other hand, $A^\Theta_{z\phi}=A^\Theta_\Theta=0$. Therefore~\eqref{eq:main2} implies that, if $\psi\in K_B$ and $A^B_\psi$ is bounded, then $A^{\Theta\circ B}_{\psi (\phi\circ B)}$ is bounded and unitarily equivalent to $A^B_\psi\otimes (\tilde k^\Theta_0\odot k^\Theta_0)$.

The following lemma is well known (and easy to prove).

\begin{lemma}\label{le:rank one}
Two operators of rank one $x_i\odot y_i\in \LL(\EE_i)$, $i=1,2$, are unitarily equivalent if and only if the following are satisfied:
\begin{enumerate}
\item $\dim\EE_1=\dim \EE_2$;
\item $\|x_1\|\|y_1\|=\|x_2\|\|y_2\|$;
\item $\<x_1, y_1\>=\<x_2, y_2\>$.
\end{enumerate}
\end{lemma}

\begin{corollary}\label{co:rank one}
Suppose that $R=x\odot y$ is a rank one operator on a Hilbert space $\EE$ with $\dim\EE\ge 2$. The following are equivalent:
\begin{itemize}
\item[{\rm (i)}]
There exists an inner function $\Theta$ such that $R$ is unitarily equivalent to a scalar multiple of $\tilde k^\Theta_0\odot k^\Theta_0$ acting in $K_\Theta$.

\item[{\rm (ii)}]
Either $R=0$ or $x, y$ are not colinear. 
\end{itemize}
\end{corollary}

\begin{proof} Suppose that (i) is satisfied. Then $\dim K_\Theta=\dim\EE\ge 2$.
 According to~\eqref{eq:reproducing kernels}, we have
 \[
 \frac{| \<\tilde k_{0}^{\Theta},  {k_{0}^{\Theta}}\> | }{\|k_{0}^{\Theta}\|\cdot \|\tilde{k_{0}^{\Theta}}\|}
 =\frac{|\tilde k_{0}^{\Theta}(0)|}{(1-|\Theta(0)|^2)^{-1}}=
 |\Theta'(0)| (1-|\Theta(0)|^2).
 \]
 The last expression can be 1 only if $\Theta(z)=z$, which contradicts $\dim K_\Theta\ge 2$. Thus it is $<1$, so $k_{0}^{\Theta}$ and $\tilde k_{0}^{\Theta}$ are not colinear. If $R\not=0$, it follows 
 from Lemma~\ref{le:rank one}  that $x$ and $y$ are not colinear.
 
 Conversely, take a nonzero operator of rank 1 $x\odot y$ acting on $\EE$; we may suppose that it has norm 1, and therefore that $\|x\|=\|y\|=1$. Then if $\Theta$ is such that $\Theta(0)=0$, $\Theta'(0)=\<x,y\>$, and $\dim K_\Theta=\dim \EE$, a second application of Lemma~\ref{le:rank one} implies that $x\odot y$ is unitarily equivalent to $\tilde k^\Theta_0\odot k^\Theta_0$.
\end{proof}

Let us also note that if $\psi$ is analytic, then $A^B_\psi=A^B_{P_\Theta\psi}$. Then
 Corollary~\ref{co:rank one}, together with the comments preceding Lemma~\ref{le:rank one}, lead to the following result.

\begin{proposition}\label{pr:rank one}
If $\psi$ is an analytic function, $A^B_\psi$ is bounded, and $R$ is a nonselfadjoint operator of rank one, then  $A^B_\psi\otimes R$ is unitarily equivalent to a TTO.
\end{proposition}

{\bf Question:} What happens if $R$ is selfadjoint? Then $A^B_\psi\otimes R$ is unitarily equivalent to a scalar multiple of $A^B_\psi\oplus 0$ (where the zero operator operates on a space whose dimension is equal to the dimension of the range of $I-R$). More about this in the next section.

\section{Direct sum with zero}

Question 5.11 of~\cite{CGRW} asks for conditions under which a direct sum of TTOs is unitarily equivalent to another TTO. We are interested here in the situation where one of the
    operators is 0 on a space of a given
    dimension $k$. The question then becomes: find classes of TTOs $A$ and  integers $k=1,2,\dots,\infty$ such that, if $0_k$ denotes the
      zero operator on a space of dimension~$k$, then  $A\oplus 0_k$ is unitarily equivalent to a TTO.

Two such classes of truncated Toeplitz operators
  have already been discovered:
\begin{enumerate}
\item $A$ a rank one operator, $k$ arbitrary (by~\cite[Theorem 5.1]{CGRW}.)
\item $A$ a normal operator, $k$ arbitrary (by~\cite[Theorem 5.7]{CGRW}.)
\end{enumerate}

Below we present two more classes of pairs $(A,k)$ that belong to the category described above. The first is obtained by using the techniques of the previous section.

\begin{theorem}\label{th:direct sum}
 Let $B$ be an inner function, $\psi \in K_B$, $\zeta\in \bbT$, and 
$n \in \bbN \cup \{\infty \}$. If $d={\dim K_B}$, then the operator 
 $A_{\psi+\zeta{\bar B}\psi}\oplus 0_{nd}$ is unitarily 
equivalent to a TTO.
\end{theorem}

\begin{proof}
Let $\Theta$ be an inner function whose angular derivative at $\zeta$ exists and such that
$\dim K_\Theta=n+1$. As noted in Section~\ref{se:preliminaries},  $k^{\Theta}_\zeta\odot k^{\Theta}_\zeta$ 
is a TTO on $K_\Theta$ with symbol  $k^{\Theta}_\zeta + \overline {k^{\Theta}_\zeta}-1$.

Let $\phi =k^{\Theta}_\zeta + \overline {k^{\Theta}_\zeta}-1+\Theta(\zeta)\bar \Theta$; then $\phi$ is also a symbol
of $k^{\Theta}_\zeta\odot k^{\Theta}_\zeta$ by ~\eqref{eq:same symbol}, and 
a simple calculation shows that 
$\phi-\bar\zeta z\phi=-\overline {\Theta (\zeta)}\Theta+\Theta(\zeta) \bar \Theta$. Applying again~\eqref{eq:same symbol}, it follows that 
$A^\Theta_\phi=\bar\zeta A^\Theta_{z\phi}=k^{\Theta}_\zeta\odot k^{\Theta}_\zeta$. So, by using \eqref{eq:main}, we obtain that 
\[
A^{\Theta\circ B}_{\psi( \phi\circ B)}\omega_B=\omega_B \left( A^B_{\psi+\zeta\bar B \psi}\otimes 
(k^{\Theta}_\zeta\odot k^{\Theta}_\zeta)\right).
\]
Since $k^{\Theta}_\zeta\odot k^{\Theta}_\zeta$ is a self adjoint operator 
 of rank one, 
$\frac{1}{\Vert k^{\Theta}_\zeta\Vert^2} A^B_{\psi+\zeta\bar B \psi}\otimes ( k^{\Theta}_\zeta\odot k^{\Theta}_\zeta)$ is unitarily
 equivalent to 
$A_{\psi+\zeta{\bar B}\psi}\oplus 0_{d}\oplus \ldots \oplus 0_{d}$, where $0_{d}$ is repeated $n$ times. Therefore
 the last operator is unitarily equivalent to the TTO  defined on $K_{\Theta\circ B}$ and with symbol
$\frac{1}{\Vert k^{\Theta}_\zeta\Vert^2} \psi \left(k^{\Theta}_\zeta\circ B + 
\overline {k^{\Theta}_\zeta\circ B}+\Theta(\zeta)\overline{\Theta\circ B}-1\right)$. 
\end{proof}

A particular case of Theorem~\ref{th:direct sum} is worth mentioning: taking $\Theta(z)=z^2$, one obtains that the operator $A_{\psi+\zeta{\bar B}\psi}\oplus 0_{d}$ is unitarily equivalent to $A^{B^2}_{\psi( \phi\circ B)}$, where $\phi (z)=\frac{1}{2}({\bar\zeta}^2{\bar z}^2+{\bar \zeta}{\bar z}+1+\zeta z)$.

\medskip

The second result in this direction is obtained without
    using the previous sections; we have added it because its fits naturally in this context. We start with a simple lemma.

\begin{lemma}\label{le:kerofTTO}
Suppose $\Theta$ is inner and $\phi\in H^\infty$. Then
$
K_\Theta\ominus \ker A^\Theta_\phi=K_u
$ 
for some inner function $u$ which satisfies
 $u|\Theta$ and $\Theta|u\phi$.
\end{lemma}

\begin{proof}
We claim first that $\ker A^\Theta_\phi \oplus \Theta H^2$ is a $z$-invariant subspace of
$H^2$. Indeed, since $ \Theta H^2$ is $z$-invariant, what is left is to see
that $z(\ker A^\Theta_\phi) \subseteq \ker A^\Theta_{\phi} \oplus \Theta H^2.$
But $f\in \ker A^\Theta_\phi$ implies that $\phi f\in \Theta H^2$, whence
$z\phi f\in \Theta H^2.$ Thus, using the fact that 
$\phi(\Theta H^2) \subseteq \Theta H^2,$ we have that:
\[
A^\Theta_\phi (P_\Theta(zf))
= P_\Theta \phi P_\Theta(zf)=P_\Theta (\phi z f- \phi P_{\Theta H^2}(zf))=0,
\]
which means that  $P_\Theta(zf)\in \ker A^\Theta_\phi$ and thus $zf\in \ker A^\Theta_\phi \oplus \Theta H^2$.

It follows then that $\ker A^\Theta_\phi \oplus \Theta H^2=u H^2$ for some inner function $u$, and thus
\[
K_\Theta\ominus \ker A^\Theta_\phi= H^2\ominus (\ker A^\Theta_\phi+\Theta H^2)=K_u. 
\]
It is  immediate that
$K_u
\subseteq K_\theta $ is equivalent to $u|\Theta$. As for the second divisibility property, note first that, since $\phi$ is analytic, $\phi(\Theta H^2)\subset \Theta H^2$. Also, if $f\in \ker A^\Theta_\phi$, then 
\[
\phi f= P_\Theta(\phi f)+P_{\Theta H^2} (\phi f)= A^\Theta_\phi f+P_{\Theta H^2} (\phi f)=P_{\Theta H^2} (\phi f)\in \Theta H^2.
\]
Since $u H^2=\ker A^\Theta_\phi \oplus \Theta H^2$, it follows that $\phi u H^2\subset \Theta H^2$, whence $\Theta|u \phi$.
\end{proof}

\begin{theorem}\label{th:nilpotent}
Suppose $\Theta$ is inner, $\phi\in H^\infty$, and $(A^\Theta_\phi)^2=0$. If $k=\dim K_\Theta\ominus \ker A^\Theta_\phi$, then $A^\Theta_\phi\oplus 0_k$ is unitarily equivalent to a TTO.
\end{theorem}

\begin{proof}
Applying  Lemma~\ref{le:kerofTTO}, we have $K_\Theta=K_u\oplus \ker A^\Theta_\phi$ for an inner function $u$ with  $u|\Theta$ and $\Theta|u\phi$. The matrix of $A^\Theta_\phi $ corresponding to 
 the orthogonal decomposition $K_\Theta= K_u\oplus \ker A^\Theta_\phi$ is
 \begin{equation}\label{eq:matrix-nilpotent}
 A^\Theta_\phi=
 \begin{pmatrix}
0 & 0\\
A &0
\end{pmatrix}.
\end{equation}

Let us consider the TTO $A^{u\Theta}_{u\phi}$ acting on $K_{u\Theta}$. We want to find its matrix according to the orthogonal decomposition  
\[
K_{u\Theta}=K_u\oplus u K_\Theta= K_u\oplus uK_u\oplus u(\ker A^\Theta_\phi).
\]
First, if $f\in K_u$, then 
\[
A^{u\Theta}_{u\phi}f=P_{{u\Theta}}( u\phi f)=
u P_\Theta (\bar u u \phi f)=u A^\Theta_\phi f.
\]
Equation~\eqref{eq:matrix-nilpotent} shows us then that the first column of the matrix of $A^{u\Theta}_{u\phi}$ is $(0\ 0\ uA)^t$.

On the other hand, $\Theta|u\phi $ implies that $u\Theta|u^2\phi$, and therefore that $u\phi(uH^2)\subset u\Theta H^2$, whence the two remaining columns of the decomposition are formed only by zeros. Therefore the matrix of  $A^{u\Theta}_{u\phi}$ is
\[
\begin{pmatrix}
0&0&0\\
0&0&0\\
uA&0&0
\end{pmatrix}.
\]
Finally, an operator unitarily equivalent to $A^{u\Theta}_{u\phi}$ is obtained through  a permutation of the two first spaces (which have the same dimension, that of $K_u$). The matrix is then  transformed into 
\[
\begin{pmatrix}
0&0&0\\
0&0&0\\
0&uA&0
\end{pmatrix},
\]
which, again by~\eqref{eq:matrix-nilpotent}, defines an operator unitarily equivalent to $A^\Theta_\phi\oplus 0_{\dim K_u}$.
\end{proof}

Let us note that there is no loss of generality in supposing $\phi\in H^\infty$, rather than $\phi\in H^2$ and $A^\Theta_\phi$ bounded. Indeed, the last assumptions imply, by the commutant lifting theorem, that $A^\Theta_\phi$ has a symbol in $H^\infty$. Also,
 a similar result obviously holds in the case where $\phi$ is antianalytic. 

\section*{Acknowledgements}
M. Zarrabi was partially supported by the ANR-09-BLAN-0058-01.
D. Timotin thanks the University of Bordeaux for its hospitality during his stay in Bordeaux, which led to the completion of this paper.

%A concrete problem: is
%\[
%\begin{pmatrix}
%1&0&0\\1&1&0\\0&0&0
%\end{pmatrix}
%\]
%unitarily equivalent to a TTO?
%

\end{document}